\documentclass{anmatuvt} 
\usepackage{mathrsfs, amssymb, eucal, amsmath, amsthm, amscd}
\usepackage{enumerate,comment}
\usepackage[T1]{fontenc}

\newtheorem{theorem}{Theorem}[section]
\newtheorem{corollary}[theorem]{Corollary}

\newtheorem{proposition}[theorem]{Proposition}

\theoremstyle{definition}
\newtheorem{definition}[theorem]{Definition}
\theoremstyle{definition}
\newtheorem{remark}[theorem]{Remark}
\theoremstyle{definition}
\newtheorem{example}[theorem]{Example}
\numberwithin{equation}{section}

\begin{document}
\title{ $(h,k)$-Dichotomy and  Lyapunov Type Norms}
\author[V. Crai]{Violeta Crai}
\address{Department of Mathematics \\
	West University of Timi\c soara \\
	V. P\^ arvan Blvd. No. 4 \\
	300223 Timi\c soara \\
	Romania} \email{vio.terlea@gmail.com}
\author[M. Megan]{Mihail Megan}
\address{Academy of Romanian Scientists,\\ Independen\c tei 54,\\ 050094 Bucure\c sti,\\  West University of Timi\c soara, \\ Faculty of Mathematics and Computer Sciences\\ V. P\^ arvan Blv. No. 4,\\ 300223 Timi\c soara,\\ Romania}
\email{mihail.megan@e-uvt.ro}
\date{ }
\maketitle
\begin{abstract}
	The paper considers a general concept of nonuniform $(h,k)$-dichotomy for evolution operators in Banach spaces. Two characterizations of this concept in terms of some families of norms compatible with the dichotomy projectors are given.
	\end{abstract}

\section{Introduction}
The notion of exponential dichotomy was introduced by O. Perron in \cite{peron} for differential equations and by T. Li in \cite{li} for difference equation. The concept plays a central role in the stability theory of differential equations, discrete dynamical systems, delay evolution equations, dynamical equations on time scales, impulsive equations, stochastic processes  and many other domains. The exponential dichotomy property for linear differential equations has gained prominence since the appearance of two fundamental monographs due to Ju.L. Dalecki and  M.G. Krein \cite{daletchi} and J.L. Massera and J.J. Sch\"{a}ffer \cite{masera}. These were followed by the book of W.A. Coppel \cite{copel}, who synthesized and improved the results that existed in the literature up to 1978. Dichotomies have been the subject of extensive research over the last years, leading to exciting new results. For more details, we refer the reader \cite{lupa}, \cite{megan2}, \cite{sasu}.

A natural generalization of both the uniform and
nonuniform, exponential and polynomial dichotomy is successfully modeled by the concept of $(h,k)$-dichotomy, where $h,k$ are growth rates (nondecreasing functions that go to  infinity). The concept was introduced for the first time in the literature by M. Pinto in \cite{pinto1} and was intensively studied in the last years (see for example \cite{Besi4}, \cite{monteola1}, \cite{megan}).

The aim of this paper is to obtain two characterizations of the notion of $(h,k)-$ dichotomy. In order to do that, we introduce the concept of a Lyapunov type family of norms $\mathcal{N}=\{\|\cdot\|_t, t\geq0\}$ compatible with the family of projectors $P$. We give two examples of these kind of norms: the first example is based on the assumption of dichotomy and the second yields by the $(h,k)$-growth property. Using these norms we obtain characterization of the concept of $(h,k)$-dichotomy. A surprising result is the equivalence between the nonuniform dichotomy property and a certain type of uniform dichotomy with respect to a Lyapunov type family of norms. We recall that, in general, the two concepts are distinguished. 

It is difficult to indicate an original reference for considering Lyapunov type families of norms in the classical uniform theory, in the nonuniform theory it first occurred in Pesin's work on nonuniform hyperbolicity and smooth ergodic theory \cite{pesin}, \cite{pesin1}. Our characterizations, using the family of norms, are inspired by the works of L. Barreira, D. Dragi\v cevi\'c and C. Valls for exponential dichotomy( see for example \cite{ba0}, \cite{ba1}, \cite{ba3}).

\section{Preliminaries}
We denote by $X$ a Banach space, $\mathcal{B}(X)$ the Banach algebra of all linear and bounded operators on $X$ and $\Delta=\{(t,s)\in\mathbb{R}_+^2,t\geq s\geq 0\}$.
\begin{definition}
	An application $U:\Delta\to\mathcal{B}(X)$ is called \textit{evolution operator} on $X$ if:
	\begin{itemize}
		\item[$(e_1)$]$U(t,t)=I$(the identity operator on $X$)
		\item[$(e_2)$]$U(t,t_0)=U(t,s)U(s,t_0)$, for all $(t,s),(s,t_0)\in\Delta$ (the evolution property)
	\end{itemize}
\end{definition}
\begin{example}
	If $X=\mathbb{R}$ and $u:\mathbb{R}_+\to\mathbb{R}^*_+$ then
	\begin{align*}
	U(t,s)=\frac{u(s)\ln u(s)}{u(t)\ln u(t)},
	\end{align*}
	for all $(t,s)\in\Delta$ is an evolution operator on $X$.
\end{example}
\begin{definition}
	A map $P:\mathbb{R}_+\to\mathcal{B}(X)$ is called \textit{a family of projectors} on $X$ if 
	$$P^2(t)=P(t),\forall t\geq 0$$
\end{definition}
\begin{example}\label{proiectori}
	If $P:\mathbb{R}_+\to\mathcal{B}(X) $ is a family of projectors on $X$ then the map $Q:\mathbb{R}_+\to\mathcal{B}(X)$ define by:
	\begin{equation*}
	Q(t)=I-P(t),\forall t\geq 0
	\end{equation*}
	is also a family of projectors on $X$, which is called the \textit{complementary family} of $P$.
\end{example}
\begin{definition}
	A family of projectors $P$ is called
	\textit{invariant} for the evolution operator $U:\Delta\to\mathcal{B}(X)$ if:
	\begin{align*}
	U(t,s)P(s)x=P(t)U(t,s)x, 
	\end{align*}
	for all $(t,s,x)\in \Delta\times X$.
\end{definition}
\begin{remark}\label{proiectori invarianti}
	If a family of projectors $P$ is invariant for the evolution operator $U:\Delta\to\mathcal{B}(X)$ then its complementary $Q$ is also invariant for $U$.
\end{remark}
Let $U:\Delta\to\mathcal{B}(X)$ be an evolution operator on $X$.
\begin{definition}
	We say that a family of projectors $P:\mathbb{R}_+\to\mathcal{B}(X)$ is
	\textit{compatible} with the evolution operator $U:\Delta\to\mathcal{B}(X)$ if it is invariant for $U$ and for every $(t,s)\in \Delta$ the restriction of $U(t,s)$ to $Ker P(s)$ is an isomorphism from $Ker P(s)$ to $Ker P(t)$.
	
\end{definition}
\begin{remark}\label{proiectorstrong}
	If the family of projectors $P:\mathbb{R}_+\to\mathcal{B}(X)$ is compatible with  the evolution operator $U:\Delta\to\mathcal{B}(X)$ and $Q:\mathbb{R}_+\to\mathcal{B}(X)$ is the complementary family of projectors of $P$, then there exists a map $V:\Delta\to\mathcal{B}(X)$ which is an isomorphism from $Ker P(t)=Range Q(t)$ to $Ker P(s)=Range Q(s)$ for all $(t,s)\in\Delta$. The isomorphism $V$ satisfies the following:
	\begin{itemize}
		\item [$v_1)$] $U(t,s)V(t,s)Q(t)x=Q(t)x$
		\item[$v_2)$] $V(t,s)U(t,s)Q(s)x=Q(s)x$
		\item[$v_3)$]$V(t,t_0)=V(s,t_0)V(t,s)$
		\item[$v_4)$]$V(t,s)Q(t)=Q(s)V(t,s)Q(t)$
	\end{itemize}
	for all $(t,s),(s,t_0)\in \Delta$ and $x\in X$.
\end{remark}
\begin{proof}
	See \cite{lupa}.
\end{proof}
\begin{definition}
	We say that a nondecreasing map $h:\mathbb{R}_+\to[1,\infty)$ is a \textit{ growth rate} if 
	\begin{align*}
	\lim\limits_{t\to\infty}h(t)=\infty
	\end{align*}
\end{definition}
\begin{example}\label{ex-rate}
	It is obvious that the functions $h_1,h_2,h_3:\mathbb{R}_+\to[1,\infty)$ defined by:
	\begin{align*}
	h_1(t)=e^t,& & h_2(t)=t+1, & &h_3(t)=(t+1)\ln(t+e),
	\end{align*}
	are growth rates.
\end{example}
Let  $h,k:\mathbb{R}_+\to[1,\infty)$ be two two growth rates and let $P:\mathbb{R}_+\to\mathcal{B}(X)$ be a family of projectors which is invariant for  the evolution operator $U:\Delta\to\mathcal{B}(X)$.
\begin{definition}\label{dicho}
	We say that the pair  $(U,P)$ is (nonuniform)-\textit{ $(h,k)$-dichotomic} if	there exists a nondecreasing map $N:\mathbb{R}_+\to[1,\infty)$ such that
	\begin{itemize}
		\item[($hd_1$)] $h(t)\|U(t,s)P(s)x\|\leq N(s)h(s) \|P(s)x\|$
		\item[($kd_2$)] $k(t)\|Q(s)x\|\leq N(t)k(s)\|U(t,s)Q(s)x\|$,
	\end{itemize}
	for all $(t,s,x)\in\Delta\times X$, where $Q$ is the complementary family of $P$.
\end{definition}
\begin{remark}
	As particularly cases of $(h,k)$-dichotomy we have:
	\begin{itemize}
		\item[1.] If N is a constant, we obtain the \textit{uniform - $(h,k)$- dichotomy} property, denoted by (u-$(h,k)$-d).
		\item[2.] Taking $h(t)=e^{ \alpha t},k(t)=e^{ \beta t}$ with $\alpha,\beta>0$ in Definition \ref{dicho}, it results the \textit{exponential dichotomy} concept, denoted by (e.d.).
		\item[4.] For $h(t)=(t+1)^\alpha,k(t)=(t+1)^\beta,\alpha,\beta>0$ in Definition \ref{dicho} we obtain the \textit{polynomial dichotomy} property, denoted by (p.d.).
	\end{itemize}
\end{remark}
\begin{remark}
If the pair $(U,P)$ is uniform -$(h,k)$- dichotomic, then it is also nonuniform -$(h,k)$- dichotomic. In general, the reverse of this statement is not valid. The following is an example of a nonuniform -$(h,k)$- dichotomy that is not uniform.	
\end{remark}
\begin{example}
	Let $h,k:\mathbb{R}_+\to\mathcal{B}(X)$ be two growth rates. On $X=\mathbb{R}^2$ endowed with the norm $\|x\|=\max\{|x_1|,|x_2|\}$, we consider the functions $P,Q:\mathbb{R}_+\to\mathcal{B}(X)$ given by
	\begin{equation}
	P(t)(x_1,x_2)=(x_1-x_2e^t,0)
	\end{equation}
	and 
	\begin{equation}
	Q(t)(x_1,x_2)=(x_2e^t,x_2).
	\end{equation}
	It is obvious that $P,Q$ are complementary and that:
	\begin{align}
	P(t)P(s)x&=P(s)x.\label{rel4}\\
	Q(t)Q(s)x&=Q(t)x,\label{rel3}
	\end{align}
	for all $(t,s,x)\in\Delta$.\\
	Further, we consider the evolution operator $U:\Delta\to\mathcal{B}(X)$ given by:
	\begin{align}
	U(t,s)= \frac{h(s)}{h(t)}\frac{(s+1)\ln(s+e) }{(t+1)\ln (t+e)}P(s)+ \frac{k(t)}{k(s)}\frac{(s+1)\ln(s+e) }{(t+1)\ln (t+e)}Q(s), 
	\end{align}for all $(t,s)\in\Delta$.\\
	 By  (\ref{rel4}) and (\ref{rel3}) we obtain that the functions $P,Q:\mathbb{R}_+\to\mathcal{B}(X)$ are complementary families of projectors invariant to the evolution operator $U:\Delta\to\mathcal{B}(X).$
	
	In the following, we will prove that the pair $(U,P)$ is nonuniform -$(h,k)$- dichotomic.	We have that:
	\begin{align*}
	\|U(t,s)P(s)x\|&=\frac{h(s)}{h(t)}\frac{(s+1)\ln(s+e)}{(t+1)\ln (t+e)}\|P(s)x\|\leq \frac{h(s)}{h(t)}(s+1)\ln (s+e)\|P(s)x\|,
	\end{align*}
	and also that:
	\begin{align*}
	\|U(t,s)Q(s)x\|&=\frac{k(t)}{k(s)}\frac{(s+1)\ln(s+e) }{(t+1)\ln (t+e)}\|Q(s)x\|\geq \frac{k(t)}{k(s)} \frac{1}{(t+1)\ln (t+e)}\|Q(s)x\|
	\end{align*}
	for all $(t,s,x)\in\Delta\times X.$

	It follows that there exists a nondecreasing function $N:\mathbb{R}_+\to[1,\infty), N(t)=(t+1)\ln (t+e)$ such that the pair $(U,P)$ is nonuniform -$(h,k)$- dichotomic.
	
	We assume that the the pair $(U,P)$ is also uniform -$(h,k)$- dichotomic and we have that there exists a constant $N\geq 1$ such that :
	\begin{align*}
	\|U(t,s)Q(s)x\|&=\frac{k(t)}{k(s)}\frac{(s+1)\ln(s+e)}{(t+1)\ln (t+e)}\|Q(s)x\|\geq \frac{1}{N}\frac{k(t)}{k(s)}\|Q(s)x\|
	\end{align*}
	Taking  $s=0$ we have:
	\begin{equation*}
	\frac{1}{(t+1)\ln (t+e)}\geq \frac{1}{N}
	\end{equation*}
	When $t\to\infty$ we obtain a contradiction. It follows that the pair $(U,P)$ is not uniform -$(h,k)$- dichotomic.
\end{example}

\begin{definition}\label{growth}
	We say that the pair $(U,P)$ has \textit{ $(h,k)$-growth} if	there exists a nondecreasing map $M:\mathbb{R}_+\to[1,\infty)$ such that:
	\begin{itemize}
		\item[($hg_1$)] $h(s)\|U(t,s)P(s)x\|\leq M(s)h(t) \|P(s)x\|$
		\item[($kg_2$)] $k(s)\|Q(s)x\|\leq M(t)k(t)\|U(t,s)Q(s)x\|$,
	\end{itemize}
	for all $(t,s,x)\in\Delta\times X$.
\end{definition}
\begin{remark}
	As particularly cases of $(h,k)-$ growth we have:
	\begin{itemize}
		\item[1.] Taking $h(t)=e^{ \alpha t},k(t)=e^{ \beta t}$ with $\alpha,\beta>0$ in Definition \ref{growth} , it results the \textit{exponential growth} concept, denoted by (e.g.).
		\item[2.] For $h(t)=(t+1)^\alpha,k(t)=(t+1)^\beta,\alpha,\beta>0$ in Definition \ref{growth}, we obtain the \textit{polynomial growth} property, denoted by (p.g.).
	\end{itemize}
\end{remark}
\begin{remark}\label{dicho-crestere}
	If the pair $(U,P)$ is $(h,k)$- dichotomic then it also has $(h,k)$- growth. The reverse is not always true, as seen from the following example.
\end{remark}
We denote by $\mathcal{G}$ the set of growth rates $h:\mathbb{R}_+\to[1,\infty)$ for which there exists a growth rate $g:\mathbb{R}_+\to[1,\infty)$ such that:
\begin{equation}\label{ex-dicho-crestere}
\frac{h^2(t)}{(t+1)\ln(t+e)}\geq g(t),\forall t\geq 0.
\end{equation}
We observe that the set $\mathcal{G}$ is not void, since the growth rate $h_3:\mathbb{R}_+\to[1,\infty)$ given by Example \ref{ex-rate} belong to $\mathcal{G}$.
\begin{example}
	We consider the Banach space $X=\mathbb{R}$, two growth rates $h,k:\mathbb{R}_+\to[1,\infty)$ with $h\in\mathcal{G}$ and a projectors family $P:\mathbb{R}_+\to\mathcal{B}(X)$ invariant to the evolution operator $U:\Delta\to\mathcal{B}(X)$ given by:
	\begin{align}
	U(t,s)=\frac{h(t)}{h(s)}\frac{(s+1)\ln(s+e)}{(t+1)\ln(t+e)}P(s)+\frac{k(s)}{k(t)}\frac{(s+1)\ln(s+e)}{(t+1)\ln(t+e)}Q(s),
	\end{align}
	for all $(t,s)\in \Delta$, where $Q$ is the complementary of $P$.
	
	It is a simple computation that there exists a nondecreasing function $M:\mathbb{R}_+\to[1,\infty)$, $M(t)=(t+1)\ln(t+e)$ such that the inequalities $(hg_1),(kg_2)$ are satisfied.
	
	We assume that the pair $(U,P)$ is also $(h,k)$- dichotomic. By Definition \ref{dicho} we have that there exists a nondecreasing function $N:\mathbb{R}_+\to[1,\infty)$ such that $(hd_1)$ and $(kd_2)$ take place.
	\begin{align*}
	\|U(t,s)P(s)x\|=\frac{h(t)}{h(s)}\frac{(s+1)\ln(s+e)}{(t+1)\ln(t+e)}\|P(s)x\|\leq N(s)\frac{h(s)}{h(t)}\|P(s)x\|,
	\end{align*}
	for all $(t,s,x)\in\Delta\times X$. Since $h\in \mathcal{G}$, we have that there exits a growth rate $g$ such that (\ref{ex-dicho-crestere}) take place. Taking  $s=0$ we obtain:
	\begin{equation*}
	g(t)\leq\frac{h^2(t)}{(t+1)\ln (t+e)}\leq N(0)
	\end{equation*}
 When $t\to\infty$ we obtain a contradiction .
\end{example}

In the particular case when the family of projectors $P$ is compatible with  the evolution operator $U$ we have:
\begin{proposition}\label{strong-dichotomy} If $P$ is compatible with  $U$ then
	the pair $(U,P)$ is $(h,k)$-dichotomic if and only if there exists a nondecreasing map $N:\mathbb{R}_+\to[1,\infty)$  such that:
	\begin{itemize}
		\item[($hd_1$')] $h(t)\|U(t,s)P(s)x\|\leq N(s)h(s) \|P(s)x\|$
		\item[($kd_2$')] $k(t)\|V(t,s)Q(t)x\|\leq N(t)k(s)\|Q(t)x\|$,
	\end{itemize}
	for all $(t,s,x)\in\Delta\times X.$
\end{proposition}
\begin{proof}
	See \cite{monteola1}.
\end{proof}
\begin{proposition}\label{crestere}
	If the family of projectors $P:\mathbb{R}_+\to\mathcal{B}(X)$ is compatible with  the evolution operator $U:\Delta\to\mathcal{B}(X)$, then the pair $(U,P)$  has $(h,k)$- growth if and only if there exists a nondecreasing map $M:\mathbb{R}_+\to[1,\infty)$ such that
	\begin{itemize}
		\item[($hg_1$')] $h(s)\|U(t,s)P(s)x\|\leq M(s)h(t) \|P(s)x\|$
		\item[($kg_2$')] $k(s)\|V(t,s)Q(t)x\|\leq M(t)k(t)\|Q(t)x\|$,
	\end{itemize}
	for all $(t,s,x)\in\Delta$, where $Q$ is the complementary family of $P$.	
\end{proposition}
\begin{proof}
	\textit{Necessity:}
	
We assume that the pair $(U,P)$ has $(h,k)$- growth and by Definition \ref{growth} we have that there exists a nondecreasing function $M:\mathbb{R}_+\to[1,\infty)$ such that $(hg_1),(kg_2)$ take place. Since the inequality ($hg_1$')  is the same as $(hg_1)$ we only have to prove ($kg_2$'). By Remark \ref{proiectorstrong} we obtain:
\begin{align*}
\|V(t,s)Q(t)x\|&=\|Q(s)V(t,s)Q(s)x\|\leq M(t)\frac{k(t)}{k(s)}\|U(t,s)V(t,s)Q(t)x\|\\
&= M(t)\frac{k(t)}{k(s)}\|Q(t)x\|,
\end{align*}
for all $(t,s,x)\in \Delta\times X.$

\textit{Sufficiency}:

We assume that there exits a nondecreasing function $M$ such that the inequalities $(hg_1'),(kg_2')$ take place. By Remark \ref{proiectorstrong} we have that:
\begin{align*}
\|Q(s)x\|&=\|V(t,s)U(t,s)Q(s)x\|=\|V(t,s)Q(t)U(t,s)Q(s)x\|\\
&\leq M(t)\frac{k(t)}{k(s)}\|Q(t)U(t,s)Q(s)x\|= M(t)\frac{k(t)}{k(s)}\|U(t,s)Q(s)x\|,
\end{align*}
for all $(t,s,x)\in \Delta\times X.$
\end{proof}
\section{ Main results}
The aim of this section is to obtain two characterizations of the concept of $(h,k)$-dichotomy, using  Lyapunov type families of norms.

Let $U:\Delta\to\mathcal{B}(X)$ be an evolution operator on $X$,  $h,k:\mathbb{R}_+\to[1,\infty)$ two growth rates  and we consider a family of projectors $P$ invariant for the evolution operator $U$.

 We introduce the concept of a Lyapunov type family of norms $\mathcal{N}=\{\|\cdot\|_t,t\geq 0\}$  compatible with the family of projectors $P $.
\begin{definition}
	A family of norms $\mathcal{N}=\{\|\cdot\|_t,t\geq 0\}$ is called \textit{compatible} with the family of projectors $P $ if there exists a nondecreasing map $N:\mathbb{R}_+\to[1,\infty)$ such that 
	\begin{equation}\label{normprop}
	\|x\|\leq \|x\|_t\leq N(t)(\|P(t)x\|+\|Q(t)x\|),
	\end{equation}
	for all $(t,x)\in\mathbb{R}_+\times X, $ where $Q$ is the complementary family of $P$.
\end{definition}
\begin{remark}
	Since $P,Q:\mathbb{R}_+\to\mathcal{B}(X)$ are complementary, replacing $x$ by $P(t)x$ respectively  by $Q(t)x$ in (\ref{normprop}), for all $t\geq 0$ we obtain that, if a family of norms $\mathcal{N}=\{\|\cdot\|_t,t\geq 0\}$ is compatible with the family of projectors $P$ then we have:
	\begin{align}
	\|P(t)x\|&\leq \|P(t)x\|_t\leq N(t)\|P(t)x\|\label{rel1}\\
	\|Q(t)x\|&\leq \|Q(t)x\|_t\leq N(t)\|Q(t)x\|\label{rel2},
	\end{align}
	for all $(t,x)\in\mathbb{R}_+\times X.$
\end{remark}
In the following we will give two examples of families of norms compatible with the family of projectors $P $.
\begin{example}\label{norma-crestere}
If the pair $(U,P)$ has $(h,k)$-growth then the following Lyapunov type family of norms $\mathcal{N}=\{\|\cdot\|_t,t\geq 0\}$:
 \begin{equation}\label{norm-crestere}
 \|x\|_t=\sup_{\tau\geq t}\frac{h(t)}{h(\tau)}\|U(\tau,t)P(t)x\|+\sup_{\tau\leq t}\frac{k(\tau)}{k(t)}\|V(t,\tau)Q(t)x\|,
 \end{equation}
 for all $(t,x)\in \mathbb{R}_+\times X$, is compatible with the family of projectors $P $.
 
 Indeed,  taking $\tau=t$ in (\ref{norm-crestere})  we have that:
 \begin{align*}
 \|x\|_t\geq \|P(t)x\|+\|Q(t)x\|\geq \|P(t)x+Q(t)x\|=\|x\|
 \end{align*}
 For the right side of (\ref{normprop}), we have from Definition \ref{crestere}, that there exists a nondecreasing function $M:\mathbb{R}_+\to\mathcal{B}(X)$ such that $(hg_1),(kg_2)$ take place and by relations (\ref{normaP}), (\ref{normQ}) we obtain:
 \begin{align*}
 \|P(t)x\|_t=\sup_{\tau\geq t}\frac{h(t)}{h(\tau)}\|U(\tau,t)P(t)x\|&\leq M(t)\|P(t)x\|,
 \end{align*}
 and
 \begin{align*}
 \|Q(t)x\|_t=\sup_{\tau\leq t}\frac{k(\tau)}{k(t)}\|V(t,\tau)Q(t)x\|&\leq M(t)\|Q(t)x\|,
 \end{align*}
 Summing the previous identities we have that:
 \begin{align*}
 \|x\|_t=\|P(t)x+Q(t)x\|_t\leq\|P(t)x\|_t+\|Q(t)x\|_t\leq M(t)(\|P(t)x\|+\|Q(t)x\|)
 \end{align*}
 In consequence, we obtain that there exists a nondecreasing function $N:\mathbb{R}_+\to[1,\infty)$, given by $N(t)=M(t)$ such that (\ref{normprop}) take place.
\end{example}

\begin{remark}
	Replacing $x$ by $P(t)x$ and respectively by $Q(t)x$ in (\ref{norm-crestere}) and by the fact that the families of projectors $P,Q$ are complementary, we obtain that the family of norms $\mathcal{N}=\{\|\cdot\|_t,t\geq 0\}$ satisfies:
	\begin{align}
	\|P(t)x\|_t&=\sup_{\tau\geq t}\frac{h(t)}{h(\tau)}\|U(\tau,t)P(t)x\|\label{normaP}\\
	\|Q(t)x\|_t&=\sup_{\tau\leq t}\frac{k(\tau)}{k(t)}\|V(t,\tau)Q(t)x\|\label{normQ},
	\end{align}
for all $(t,x)\in\mathbb{R}_+\times X$.	
\end{remark}
\begin{example}\label{norma-dichotomie}
If the pair $(U,P)$ is $(h,k)$- dichotomic then the family of norms $\mathcal{N}_1=\{\||\cdot\||_t, t\geq0\}$ given by :
\begin{equation}\label{norm}
\||x\||_t=\sup_{\tau\geq t}\frac{h(\tau)}{h(t)}\|U(\tau,t)P(t)x\|+\sup_{\tau\leq t}\frac{k(t)}{k(\tau)}\|V(t,\tau)Q(t)x\|
\end{equation}
for all $(t,x)\in\mathbb{R}_+\times X$, 
is compatible with the family of projectors $P.$

	In order to obtain the left side of (\ref{normprop}) we take $\tau=t$ in (\ref{norm}) and we have that:
	\begin{align*}
	\||x\||_t\geq \|P(t)x\|+\|Q(t)x\|\geq \|P(t)x+Q(t)x\|=\|x\|
	\end{align*}
	For the right side of (\ref{normprop}) we have from Proposition \ref{strong-dichotomy}, that there exists a nondecreasing function $N:\mathbb{R}_+\to\mathcal{B}(X)$  such that $(hd_1'),(kd_2')$ take place and by relations (\ref{proiector1}), (\ref{proiector2}) we obtain:
	\begin{align*}
	\||P(t)x\||_t=\sup_{\tau\geq t}\frac{h(\tau)}{h(t)}\|U(\tau,t)P(t)x\|&\leq N(t)\|P(t)x\|,
	\end{align*}
	and
	\begin{align*}
	\||Q(t)x\||_t=\sup_{\tau\leq t}\frac{k(t)}{k(\tau)}\|V(t,\tau)Q(t)x\|&\leq N(t)\|Q(t)x\|,
	\end{align*}
	From the previous identities we have that:
	\begin{align*}
	\||x\||_t=\||P(t)x+Q(t)x\||_t\leq\||P(t)x\||_t+\||Q(t)x\||_t\leq N(t)(\|P(t)x\|+\|Q(t)x\|).
	\end{align*}
\end{example}
\begin{remark}
		Since the families of projectors $P,Q$ are complementary, replacing $x$ by $P(t)x$ and respectively by $Q(t)x$ in (\ref{norm}) we obtain that the family of norms $\mathcal{N}_1$ satisfies the following identities:
		\begin{align}
		\||P(t)x\||_t&=\sup_{\tau\geq t}\frac{h(\tau)}{h(t)}\|U(\tau,t)P(t)x\|\label{proiector1}\\
		\||Q(t)x\||_t&=\sup_{\tau\leq t}\frac{k(t)}{k(\tau)}\|V(t,\tau)Q(t)x\|\label{proiector2},
		\end{align}
		for all $(t,x)\in\mathbb{R}_+\times X.$
\end{remark}
Our first result establishes the equivalence between the notions of  nonuniform and a certain type of uniform -$(h,k)$- dichotomy with respect to a Lyapunov type family of norms.
\begin{theorem}\label{unif=neunif}Let $P:\mathbb{R}_+\to\mathcal{B}(X)$ be a family of projectors compatible with $U$.
	The pair $(U,P)$ is  $(h,k)$-dichotomic if and only	if there exists a family of norms $\mathcal{N}_1=\{\||\cdot\||_t, t\geq0\}$ compatible with the family of projectors $P $ such that the following inequalities take place:
	\begin{itemize}
		\item[($hd_1"$)]$h(t)\||U(t,s)P(s)x\||_t\leq h(s)\||P(s)x\||_s$
		\item[($kd_2"$)]$k(t)\||V(t,s)Q(t)x\||_s\leq k(s)\||Q(t)x\||_t$,
	\end{itemize}
	for all $(t,s,x)\in\Delta\times X$.
\end{theorem}
\begin{proof} 
	\textit{Necessity:}
	
	We assume that the pair $(U,P)$ is $(h,k)$- dichotomic. In this case, by Example \ref{norma-dichotomie} we have that there exists a family of norms $\mathcal{N}_1=\{\||x\||_t,t\geq 0\}$, given by (\ref{norm}) compatible with the family of projectors $P $. We only have to prove  the ($hd_1"$),($kd_2"$) inequalities. In order to do that, we use the relations (\ref{proiector1}), (\ref{proiector2}) and  the fact that $[t,\infty)\subseteq[s,\infty)$ and $[0,s]\subseteq[0,t]$, for all $(t,s)\in\Delta.$
	\begin{align*}
	\||U(t,s)P(s)x\||_t&=\||P(t)U(t,s)P(s)x\||_t=\sup_{\tau\geq t}\frac{h(\tau)}{h(t)}\|U(\tau,t)P(t)U(t,s)P(s)x\|\\
	&\leq \sup_{\tau\geq s}\frac{h(\tau)}{h(t)}\|U(\tau,s)P(s)x\|\\
	&=\frac{h(s)}{h(t)} \sup_{\tau\geq s}\frac{h(\tau)}{h(s)}\|U(\tau,s)P(s)x\|\\
	&=\frac{h(s)}{h(t)}\||P(s)x\||_s
	\end{align*}
	\begin{align*}
	\||V(t,s)Q(t)x\||_s&=\||Q(s)V(t,s)Q(t)x\||_s=\sup_{\tau\leq s} \frac{k(s)}{k(\tau)}\|V(s,\tau)Q(s)V(t,s)Q(t)x\|\\
	&\leq \sup_{\tau\leq t}\frac{k(s)}{k(\tau)}\|V(t,\tau)Q(t)x\|=\frac{k(s)}{k(t)}\sup_{\tau\leq t}\frac{k(t)}{k(\tau)}\|V(t,\tau)Q(t)x\|\\
	&=\frac{k(s)}{k(t)}\||Q(t)x\||_t
	\end{align*}
	for all $(t,s,x)\in\Delta\times X.$
	
	\textit{Sufficiency}:
	
	We assume that there exists a Lyapunov type family of norms $\mathcal{N}_1=\{\||\cdot\||_t, t\geq0 \}$, compatible with the family of projectors $P$ such that ($hd_1"$),($kd_2"$) take place.
	
	The implication 
	($hd_1"$)$\Rightarrow(hd_1$') yields by (\ref{rel1}).
	\begin{align*}
	\|U(t,s)P(s)x\|&\leq \||U(t,s)P(s)x\||_t\leq  \frac{h(s)}{h(t)}\||P(s)x\||_s\\
	&\leq N(s)\frac{h(s)}{h(t)}\|P(s)x\|
	\end{align*}
	for all $(t,s,x)\in\Delta\times X$.\\
	Similarly, for the implication  
	($kd_2"$)$\Rightarrow(kd_2$') we have:
	\begin{align*}
	\|V(t,s)Q(t)x\|&\leq\||V(t,s)Q(t)x\||_s\leq \frac{k(s)}{k(t)}\||Q(t)x\||_t\\
	&\leq N(t)\frac{k(s)}{k(t)}\|Q(t)x\| 
	\end{align*}
	for all $(t,s,x)\in\Delta\times X.$ 
	In conclusion, the pair $(U,P)$ is $(h,k)$-dichotomic.
\end{proof}

In the particular case, when $h(t)=e^{\alpha(t)}$ and $k(t)=e^{\beta t}$ for all $t\geq 0$ and $\alpha,\beta>0$, we obtain the following characterization of  exponential dichotomy.
\begin{corollary}\label{unif=neunif-exp}Let $P:\mathbb{R}_+\to\mathcal{B}(X)$ be a  family of projectors compatible with $U$.
	The pair $(U,P)$ is exponentially dichotomic if and only	if there exist a family of norms $\mathcal{N}_1=\{\||\cdot\||_t, t\geq0\}$ compatible with the family of projectors $P $ and two positive constants $\alpha,\beta$ such that the following inequalities take place:
	\begin{itemize}
		\item[($ed_1$)]$\||U(t,s)P(s)x\||_t\leq e^{-\alpha(t-s)}\||P(s)x\||_s$
		\item[($ed_2$)]$\||V(t,s)Q(t)x\||_s\leq e^{-\beta(t-s)}\||Q(t)x\||_t$,
	\end{itemize}
	for all $(t,s,x)\in\Delta\times X$.	
\end{corollary}
When the growth rates are of polynomial type, $h(t)=(t+1)^\alpha,k(t)=(t+1)^\beta$, for all $t\geq 0$ and $\alpha,\beta>0$, we obtain a characterization of nonuniform polynomial dichotomy in terms of uniform polynomial dichotomy with respect to a Lyapunov type family of norms.
\begin{corollary}\label{unif=neunif-pol}Let $P:\mathbb{R}_+\to\mathcal{B}(X)$ be a family of projectors compatible with  $U$.
	The pair $(U,P)$ is  polynomially dichotomic if and only	if there exist a family of norms $\mathcal{N}_1=\{\||\cdot\||_t, t\geq0\}$ compatible with the family of projectors $P $ and two positive constants $\alpha,\beta$ such that the following inequalities take place:
	\begin{itemize}
		\item[($pd_1$)]$(t+1)^\alpha\||U(t,s)P(s)x\||_t\leq (s+1)^\alpha\||P(s)x\||_s$
		\item[($pd_2$)]$(t+1)^\beta\||V(t,s)Q(t)x\||_s\leq (s+1)^\beta\||Q(t)x\||_t$,
	\end{itemize}
	for all $(t,s,x)\in\Delta\times X$.
\end{corollary}

Our second result is a characterization of the concept of $(h,k)$-dichotomy with respect to a Lyapunov type  family of norms.
\begin{theorem}Let $P:\mathbb{R}_+\to\mathcal{B}(X)$ be a family of projectors compatible with $U$ and the pair $(U,P)$ has $(h,k)$- growth.
	The pair $(U,P)$ is $(h,k)$-dichotomic if and only if there exist a family of norms $\mathcal{N}=\{\|\cdot\|_t, t\geq0\}$ compatible with $P$ and a nondecreasing function $N:\mathbb{R}_+\to[1,\infty)$ such that the following inequalities take place:
	\begin{itemize}
		\item[($hd_1$''')]$h(t)\|U(t,s)P(s)x\|_t\leq N(s)h(s)\|P(s)x\|_s$
		\item[($kd_2$''')]$k(t)\|V(t,s)Q(t)x\|_s\leq N(t)k(s)\|Q(t)x\|_t$,
	\end{itemize}
	for all $(t,s,x)\in\Delta\times X$.
\end{theorem}
\begin{proof}
\textit{Necessity}:

We assume that the pair $(U,P)$ is $(h,k)$-dichotomic and we have by Remark \ref{dicho-crestere} that the pair $(U,P)$ has $(h,k)$- growth and by Example \ref{norma-crestere}, there exists a  family of norms $\mathcal{N}=\{\|\cdot\|_t, t\geq0\}$, given by (\ref{norm-crestere}) compatible with $P$.

From Proposition \ref{strong-dichotomy} we obtain that there exits a nondecreasing function $N:\mathbb{R}_+\to[1,\infty)$ such that $(hd_1')$ and $(hd_2')$ take place.
	By  the fact that the family of norms $\mathcal{N}$ is compatible with the family of projectors $P $ we have:
	\begin{align*}
	\|U(t,s)P(s)x\|_t&=\|P(t)U(t,s)P(s)x\|_t=\sup_{\tau\geq t}\frac{h(t)}{h(\tau)}\|U(\tau,t)P(t)U(t,s)P(s)x\|\\
	&\leq \sup_{\tau\geq t}\frac{h(t)}{h(\tau)}N(s)\frac{h(s)}{h(\tau)}\|P(s)x\|= \sup_{\tau\geq t}\left( \frac{h(t)}{h(\tau)}\right) ^2N(s)\frac{h(s)}{h(t)} \|P(s)x\|\\
	&\leq N(s)\frac{h(s)}{h(t)} \|P(s)x\|\\
	&\leq N(s)\frac{h(s)}{h(t)}\|P(s)x\|_s
	\end{align*}
	and
	\begin{align*}
	\|V(t,s)Q(t)x\|_s&=\|Q(s)V(t,s)Q(t)x\|_s=\sup_{\tau\leq s} \frac{k(s)}{k(\tau)}\|V(s,\tau)Q(s)V(t,s)Q(t)x\|\\
	&\leq \sup_{\tau\leq s}\frac{k(s)}{k(\tau)}N(t)\frac{k(\tau)}{k(t)}\|Q(t)x\|=\sup_{\tau\leq s}N(t)\frac{k(s)}{k(t)}\|Q(t)x\|\\
	&\leq N(t)\frac{k(s)}{k(t)}\|Q(t)x\|_t
	\end{align*}
	for all $(t,s,x)\in \Delta\times X.$

	\textit{Sufficiency}:
	
	We assume that there exist a family of norms $\mathcal{N}=\{\|\cdot\|_t,t\geq 0\}$ compatible with $P$ and a nondecreasing function $N:\mathbb{R}_+\to[1,\infty)$ such that the inequalities ($hd_1$'''),($kd_2$''') take place. By relation (\ref{rel1}) and (\ref{rel2}) we have:
	\begin{align*}
	\|U(t,s)P(s)x\|&=\|P(t)U(t,s)P(s)x\|\leq \|U(t,s)P(s)x\|_t\\
	&\leq N(s)\frac{h(s)}{h(t)}\|P(s)x\|_s\leq N(s)M(s)\frac{h(s)}{h(t)}\|P(s)x\|
	\end{align*}
	and
	\begin{align*}
	\|V(t,s)Q(t)x\|&=\|Q(s)V(t,s)Q(t)x\|\leq \|V(t,s)Q(t)x\|_s\\
	&\leq N(t)\frac{k(s)}{k(t)}\|Q(t)x\|_t \leq N(t)M(t)\frac{k(s)}{k(t)}\|Q(t)x\|
	\end{align*}
	for all $(t,s,x)\in\Delta\times X$. In conclusion, we obtain that there exists a nondecreasing function $N_1:\mathbb{R}_+\to[1,\infty)$ given by $N_1(t)=N(t)M(t)$ for all $t\geq 0$, such that inequalities $(hd_1')$ and $(hd_2')$ take place. Thus, the pair $(U,P)$ is $(h,k)$- dichotomic.
\end{proof}
In the particular cases, when the growth rates $h,k:\mathbb{R}_+\to[1,\infty)$ are of exponential and polynomial type, we obtain the following characterizations of  exponential respectively, polynomial dichotomy, in terms of Lyapunov type families of norms.
\begin{corollary}Let $P:\mathbb{R}_+\to\mathcal{B}(X)$ be a family of projectors compatible with $U$ and the pair $(U,P)$ has exponential growth.
	The pair $(U,P)$ is exponentially dichotomic if and only if there exist a family of norms $\mathcal{N}=\{\|\cdot\|_t,t\geq 0\}$ compatible with $P$, a nondecreasing function $N:\mathbb{R}_+\to[1,\infty)$ and two positive constants $\alpha,\beta$, such that the following inequalities take place:
	\begin{itemize}
		\item[$(ed_1')$]$\|U(t,s)P(s)x\|_t\leq N(s)e^{-\alpha(t-s)}\|P(s)x\|_s$
		\item[$(ed_2')$]$\|V(t,s)Q(t)x\|_s\leq N(t)e^{-\beta(t-s)}\|Q(t)x\|_t$,
	\end{itemize}
	for all $(t,s,x)\in\Delta\times X$.
\end{corollary}
\begin{corollary}
	Let $P:\mathbb{R}_+\to\mathcal{B}(X)$ be a family of projectors compatible with $U$ and the pair $(U,P)$ has polynomial growth.
	The pair $(U,P)$ is polynomially dichotomic if and only if there exist a family of norms  $\mathcal{N}=\{\|\cdot\|_t,t\geq 0\}$ compatible with $P$, a nondecreasing function $N:\mathbb{R}_+\to[1,\infty)$ and the positive constants $\alpha,\beta$, such that the following inequalities take place:
	\begin{itemize}
		\item[$(pd_1')$]$(t+1)^\alpha\|U(t,s)P(s)x\|_t\leq N(s)(s+1)^\alpha\|P(s)x\|_s$
		\item[$(pd_2')$]$(t+1)^\beta\|V(t,s)Q(t)x\|_s\leq N(t)(s+1)^\beta\|Q(t)x\|_t$,
	\end{itemize}
	for all $(t,s,x)\in\Delta\times X$.
\end{corollary}

\end{document}